\documentclass{amsart}
\usepackage[colorlinks]{hyperref} %pdflatex
\usepackage{microtype}
\usepackage{graphicx}

\newbox\mybox
\def\overtag#1#2#3{\setbox\mybox\hbox{$#1$}\hbox to
  0pt{\vbox to 0pt{\vglue-#3\vglue-\ht\mybox\hbox to \wd\mybox
      {\hss$\ss#2$\hss}\vss}\hss}\box\mybox}
\def\undertag#1#2#3{\setbox\mybox\hbox{$#1$}\hbox to 0pt{\vbox to
    0pt{\vglue#3\vglue\ht\mybox\hbox to \wd\mybox
      {\hss$\ss#2$\hss}\vss}\hss}\box\mybox}
\def\lefttag#1#2#3{\hbox to 0pt{\vbox to 0pt{\vss\hbox to
      0pt{\hss$\ss#2$\hskip#3}\vss}}#1}
\def\righttag#1#2#3{\hbox to 0pt{\vbox to 0pt{\vss\hbox to
      0pt{\hskip#3$\ss#2$\hss}\vss}}#1}
\let\ss\scriptstyle

\def\Dot{\lower.2pc\hbox to 2.5pt{\hss$\bullet$\hss}}
\def\Circ{\lower.2pc\hbox to 2.5pt{\hss$\circ$\hss}}
\def\Vdots{\raise5pt\hbox{$\vdots$}}
\def\splicediag#1#2{\xymatrix@R=#1pt@C=#2pt@M=0pt@W=0pt@H=0pt}
\newcommand\lineto{\ar@{-}}
\newcommand\dashto{\ar@{--}}
\newcommand\dotto{\ar@{.}}

\newcommand{\Z}{{\mathbb Z}}
\newcommand{\R}{{\mathbb R}}

\newcommand{\co}{\colon}

\usepackage{color}

\DeclareMathOperator{\CAT}{CAT}
\newcommand\catzero{$\CAT(0)$ }

\newtheorem{theorem}{Theorem}[section]
\newtheorem*{theorem*}{Theorem}

\newtheorem{lemma}[theorem]{Lemma}
\newtheorem{proposition}[theorem]{Proposition}
\newtheorem{corollary}[theorem]{Corollary}
\newtheorem*{corollary*}{Corollary}
\theoremstyle{definition}
\newtheorem{definition}[theorem]{Definition}
\newtheorem{example}[theorem]{Example}
\newtheorem{remark}[theorem]{Remark}

\newcommand{\ulp}{\underline{\bf p}}
\newcommand{\ulf}{\underline{\bf f}}

\def\p{{\bf p}}
\def\f{{\bf f}}

\def\TT{\mathcal T}
\def\TTn{\TT_{n}}

\def\coneA{A_{\omega}}

\begin{document}
\title [QI classification of some high dimensional 
right-angled Artin groups]
{Quasi-isometric classification of some\\ high dimensional 
right-angled Artin groups}


\author
{Jason A. Behrstock}
\address{Department of Mathematics\\Lehman College, CUNY} 
\email{\href{mailto:jason.behrstock@lehman.cuny.edu}{jason.behrstock@lehman.cuny.edu}}
\author{Tadeusz Januszkiewicz}\address{Department of Mathematics, Ohio
  State University} 
\email{\href{mailto:tjan@math.osu.edu}{tjan@math.osu.edu}}\author{Walter
D.
  Neumann} \address{Department of Mathematics\\Barnard College,
  Columbia University}
\email{\href{mailto:neumann@math.columbia.edu}{neumann@math.columbia.edu}}
\date{today: \today }
\begin{abstract}
In this note we give the quasi-isometry classification for a class of right
  angled Artin groups. In particular, we obtain the first such
  classification for a class of Artin groups with dimension larger
  than 2; our families exist in every dimension.
\end{abstract}

\maketitle

\section{Introduction}
\subsection{Background}

A \emph{right-angled Artin group} is a finitely presented group $G$
which can be described by a finite graph $\Gamma$, the 
\emph{presentation graph},  in the following
way: the vertices of $\Gamma$ are in bijective correspondence with the
generators of $G$ and the defining relations in $G$ consist of a
commuting relation between each pair of generators connected by  an
edge in $G$.
Right-angled Artin groups interpolate between 
free groups (defined by  graphs with no edges) and 
free abelian groups (defined by  complete graphs). 
In between these two extremes, right-angled Artin groups include 
a rich source of interesting groups. In this paper we will describe 
the quasi-isometric classification of a family of such groups.

The two main families of right-angled Artin groups which have been
classified are those whose presentation graphs are trees or
\emph{atomic}. It was proven by Behrstock and Neumann
\cite{BehrstockNeumann:qigraph} that all right-angled Artin groups
which have a presentation graph a tree of diameter greater than two
are quasi-isometric to each other and are not quasi-isometric to any
other right-angled Artin groups; trees of diameter two give the
product of a nonabelian free group with an infinite cyclic group and
these are all quasi-isometric to each other and to no other right
angled Artin group by work of Kapovich and Leeb
\cite{KapovichLeeb:haken}; the tree of
diameter 1 corresponds to $\Z^{2}$, which is not quasi-isometric to
any other
right-angled Artin group. Atomic graphs were
introduced by Bestvina, Kleiner and Sageev; these are connected graphs
with no valence one vertices, no cycles of length less than five, and
no separating closed vertex stars; they proved that right-angled
Artin groups with
presentation graphs that are atomic are quasi-isometric if and only if
the groups have isomorphic presentation graphs
\cite{BestvinaKleinerSageev:RAAG1}. Note that both trees and 
atomic graphs yield Artin groups with cohomological dimension at most 2
since the cohomological dimension of the group 
is the number of vertices of a maximal complete subgraph
(cf.\ \cite{CharneyDavis:Finite}).

The only other family of connected right-angled Artin groups we are
aware of which is completely classified is given by complete graphs;
this follows since $\Z^{n}$ is quasi-isometric to $\Z^{m}$ if and only
if $n=m$. Since all other right-angled Artin groups have 
free subgroups it follows that these groups are not 
quasi-isometric to any other right-angled Artin group.

\subsection{Results}

Define $\mathcal T_n$ to be the smallest class of $n$-dimensional
simplicial complexes satisfying:
\begin{itemize}
\item The $n$--simplex is in $\mathcal T_n$;
\item If $K_1$ and $K_2$ are complexes in $\mathcal T_n$ then the
  union of $K_1$ and $K_2$ along any $(n-1)$--simplex is in $T_n$. 
\end{itemize}
For $n=1$ this is the class of finite trees.  
For $K\in \mathcal T_n$ let $A_K$ denote the
right-angled Artin group whose presentation graph is the
$1$--skeleton of $K$,  we call this a 
\emph{right-angled $n$--tree group}. (Note that $\Z^3$ and
the right-angled $1$--tree groups are exactly the right-angled 
Artin groups which are the fundamental groups of compact 
$3$--manifolds, 
\cite{HermillerMeier}.)
If
$K$ has a vertex that is distance $1$ from all other vertices, then it
is the cone on some $K'\in \mathcal T_{n-1}$ and hence $A_K\cong
\Z\times
A_{K'}$; we say that such an $A_K$ is \emph{reducible}. 

To each $K\in \mathcal T_n$ we associate a tree $\Gamma(K)$ with a 
vertex-coloring
in a way to be described in section \ref{sec:2}. The colors consist of
$n+1$ ``\p--colors'' and one ``\f-color''. In that section we also
describe a ``bisimilarity'' relation, as used in
\cite{BehrstockNeumann:qigraph}, for such trees.

The following is our main result, which is proven 
in sections \ref{sec:3} and \ref{sec:4}. This gives the first
non-trivial 
classification theorem of high dimensional right-angled Artin groups.

\begin{theorem}\label{qintrees} 
  Given $K,K'\in \mathcal T_{n}$. The groups $A_{K}$ and $A_{K'}$ are
  quasi-isometric if and only if $\Gamma(K)$ and $\Gamma(K')$ are
  bisimilar after possibly reordering the \p--colors
by an
  element of the symmetric group on $n+1$ elements.
\end{theorem}

As an immediate consequence we obtain the following, which generalizes
\cite[Theorem~3.2]{BehrstockNeumann:qigraph}, where the $n=1$ case was
established. We define an element
$K\in\mathcal T_n$ to be \emph{maximally branched} if each
$n$--simplex has
other simplices glued to it either along exactly one $(n-1)$--face or
along all of its $(n-1)$--faces; we say that $A_{K}$ 
is \emph{maximally branched} if $K$ is maximally branched.

\begin{corollary}\label{maxbranchedntreesqi} For any fixed $n$, 
  any two irreducible maximally branched right angle $n$--tree groups
  are quasi-isometric.
\end{corollary}

A consequence of Theorem~\ref{qintrees} together with a theorem of 
Papasoglu and Whyte concerning quasi-isometric invariance of 
free product decompositions \cite[Theorem~0.4]{PapasogluWhyte:ends} 
is the following:

\begin{corollary} Let $\mathcal K=\{K_{1},\ldots, K_{n}\}$ and 
    $\mathcal K'=\{K'_{1},\ldots, K'_{m}\}$ be finite sets 
    of elements with $K_{i}\in \mathcal T_{n(i)}$ and $K'_{j}\in \mathcal T_{n(j)}$.
    Let $A_{\mathcal K}$ 
    be the right-angled Artin group whose presentation graph is the 
    disjoint union of the 1--skeleton of the $K_{i}$, 
    define $A_{\mathcal K'}$ similarly. 
    
    Then the group $A_{\mathcal 
    K}$  is quasi-isometric to $A_{\mathcal K'}$ if and only if for 
    each $K_{i}$ there exists $j$ with $n(i)=n(j)$ and 
    $\Gamma(K_{i})$ bisimilar to $\Gamma(K'_{j})$ and for each $K'_{p}$ 
    there exists $K_{q}$ with $n(p)=n(q)$ and 
    $\Gamma(K'_{p})$ bisimilar to $\Gamma(K_{q})$.
\end{corollary}

In a previous paper, we showed
that in the case where all the $K_{i}$ and $K_{j}$ are simplicies, 
then the quasi-isometric classification of free products agrees with 
the commensurability classifications
\cite{BehrstockNeumannJanuszkiewicz:freeproductsabelian}. Already in 
the class of groups $\{A_K:K\in\mathcal T_{1}\}$ are infinite families of  
quasi-isometric, but pairwise non-commensurable groups. A question 
that remains open is to find the commensurabilty classification of 
the groups discussed here. 
In the remainder of the paper, unless we specify otherwise, we will only 
consider connected presentation graphs.

\subsection*{Acknowledgements} We thank the anonymous referees for 
useful comments; in particular for the suggestion of adding 
Remark~\ref{remark:labels}.

\section{Preliminaries}\label{sec:2}

\subsection{Geometric models}

We describe the geometric models that we will work with.  Fix a
complex $K\in \mathcal T_n$. 
We define a
\emph{piece} to be the star in $K$ of an $(n-1)$--simplex of $K$
which is the boundary of at least 2 $n$--simplices.  Let
$P$ denote a piece of $K$.  Then, $P$ consists of a finite collection
of $n$--simplices attached along the common $(n-1)$--simplex, i.e.,
the join of the $(n-1)$--simplex with a finite set of points
$p_1,\dots,p_k$.  The Artin group $A_P$ is thus the product of a free
group of rank $k$ with $\Z^{n}$.  Giving the free group the redundant
presentation
$$\langle p_0,p_1,\dots,p_k:p_0p_1\dots p_k=1\rangle$$
allows us to naturally think of it as the fundamental group of a
$(k+1)$--punctured sphere $S_{k+1}$.  Hence, $A_P$ is the fundamental
group of $M=S_{k+1}\times T^n$, with the $k$ $n$--simplices of $P$
representing the fundamental groups of $k$ of the $k+1$ boundary
components.

When two pieces $P$ and $P'$ of $K$ intersect in an $n$--simplex this
corresponds to gluing the corresponding manifolds, $M$ and $M'$, along
a boundary component by a \emph{flip} --- a map that switches the base
coordinate of one piece with one of the $S^1$ factors of the torus
fiber of the other piece.  Since the torus has $n+1$ factors $S^1$,
there are $n$ possible flips we can use for such a gluing.  
In this way we associate to any complex $K\in
\mathcal T_n$ a space $X_K$ with fundamental group $A_{K}$ which is a
manifold away from a certain ``branch locus''.  This branch locus
consists of the collection of $(n+1)$--tori corresponding to
$n$--simplices in $K$ which are contained in more than two
pieces. Note that for $n=1$ the branch locus is always empty, whereas
for $n>1$ it is empty if and only if every $n$--simplex is contained
in at most two pieces: such \emph{minimally branched} complexes yield 
a family of ``high dimensional graph manifolds'' (i.e., manifolds glued 
from trivial bundles of tori over compact surfaces with boundary) 
which are quasi-isometrically classified as a special case of 
Theorem~\ref{qintrees}.

We call the decomposition of $X_K$ into its pieces
the \emph{geometric decomposition}. There is a corresponding
graph-of-groups
decomposition of $A_K$ with two kinds of vertex groups, the
fundamental
groups of the pieces and the fundamental groups of the separating
tori; the edge groups are copies of the fundamental groups of the
separating tori, one copy for each geometric piece that the torus
bounds.

\begin{remark} For a complex $K\in
  \mathcal T_n$ and for any piece $P$ as above, $A_P$ is
  quasi-isometrically embedded in $A_{K}$. This holds since there
  exists a retraction from $A_{K}$ to $A_{P}$, cf.\ \cite[Proposition
  10.4]{BehrstockDrutuMosher:thick}. (This is more generally true for
  any full subcomplex.)
\end{remark}

\subsection{Labelled graphs}\label{sec: Labelled graphs}

To each $K\in \mathcal T_n$ we will associate a labelled bipartite
tree, $\Gamma(K)$, whose underlying graph is the graph of the
graph-of-groups decomposition of $A_K$ described above.

To each piece in $K$ we assign a vertex labelled \p\ (for \ulp
iece). To each of the $n$--simplices 
which is in more than one piece we assign a vertex labelled \f\
(for \ulf ace). Each \f--vertex is connected by an edge to each of
the \p--vertices which corresponds to a piece containing the
$n$--simplex.

Since for any $K\in \mathcal T_n$ there is simplicial map to an 
$n$--dimensional simplex $\Delta$, which is unique up to permutation 
of $\Delta$, it follows that labelling the vertices of $\Delta$ by 
$1$ to $n+1$ pulls back to a consistent labelling on all the vertices 
of $K$. 
Note that in any piece all 
the vertices of their common $(n-1)$--simplex (the ``spine'' of the
piece) are given the same label.
We label each \p--vertex by the index of the $n$-simplex vertex
which is not on the spine of the corresponding piece. Hence the label 
set for the \p--vertices are the elements of 
the set $\{1,\dots,n+1\}$. The
possible labels for vertices are thus \p1, \p2, \dots, \p$(n+1)$ and
\f,
for a total of $n+2$ possible labels.

The \p/\f--distinction gives a bipartite structure on our tree
$\Gamma(K)$. The \p--vertices to which a given \f--vertex is connected
have distinct labels, so a \f--vertex has valence at most $n+1$ (and
at least $2$). A \p--vertex can be connected to any number of
\f--vertices.

\subsection{Bisimilarity}

\begin{definition}\label{def:color}
  A \emph{graph} $\Gamma$ consists of a vertex set $V(\Gamma)$ and an
  edge set $E(\Gamma)$ with a map $\epsilon\co E(\Gamma)\to
  V(\Gamma)^2/\Z_2$ to the set of unordered pairs of
  elements of $V(\Gamma)$.

 A \emph{colored graph} is a graph $\Gamma$, a set $C$, and a ``vertex
coloring'' 
$c\co V(\Gamma) \to C.$

  A \emph{weak covering} of colored graphs is a graph homomorphism
  $f\co \Gamma \to \Gamma' $ which respects colors and has the
  property: for each $v\in V(\Gamma)$ and for each edge $e'\in
  E(\Gamma')$ at $f(v)$ there exists an $e\in E(\Gamma)$ at $v$ with
  $f(e)=e'$.
\end{definition}

Henceforth, we assume all graphs we consider to be
connected. It is easy to see that a weak covering is then
surjective. 

\begin{definition}\label{def:bisimilar}
  Colored graphs $\Gamma_1,\Gamma_2$ are
  \emph{bisimilar}, written $\Gamma_1\sim\Gamma_2$, if $\Gamma_1$ and
  $\Gamma_2$ weakly cover some common colored graph.
\end{definition}

Our applications of bisimilarity rely on the following.

\begin{proposition}[\cite{BehrstockNeumann:qigraph}]\label{prop:bisimeq}
  The bisimilarity relation $\sim$ is an equivalence relation,
  and each equivalence class has a unique minimal element up to
  isomorphism.\qed
\end{proposition}

The following also holds, with a proof as in
\cite{BehrstockNeumann:qigraph} .

\begin{proposition} \label{prop:bicolor tree} If we restrict to
  connected bipartite colored graphs of the type in the previous
  subsection (\p/\f--bipartite, and the \p--vertices attached to an
  \f--vertex have distinct colors from the set $\{1,\dots,n+1\}$),
  which are countable but may be infinite, then each bisimilarity
  class contains a tree $T$, unique up to isomorphism, which weakly
  covers every element of the class.  It can be constructed as
  follows: If\/ $\Gamma$ is in the bisimilarity class, duplicate every
  \f--vertex and its adjacent edges a countable infinity of times, and
  then take the universal cover of the result (in the topological
  sense).\qed
\end{proposition}

\subsection{Examples}

\begin{example}In Figure~\ref{Figure 1} we give three minimally 
    branched complexes 
    $K_{1},K_{2},K_{3}\in \mathcal T_{2}$ and their associated 
    labelled graphs $\Gamma(K_{i})$. Notice that $\Gamma(K_{1})$ is 
    easily checked to be 
    minimal. It is also easy to check that 
    $\Gamma(K_{2})$ weakly covers $\Gamma(K_{1})$ 
    by sending both the $\p2$ vertices together and the $\p1$ 
    vertices together and hence $\Gamma(K_{1})$ is the minimal graph 
    in the bisimilarity class of $\Gamma(K_{2})$. 
    On the other hand, the graph $\Gamma(K_{3})$ 
    is minimal and hence not bisimilar to either of the other two 
    graphs. See \cite{BehrstockNeumann:qigraph} for an algorithm 
    to determine minimality.
\end{example}

\begin{figure}[ht]
    \centering
\includegraphics[width=.7\hsize]{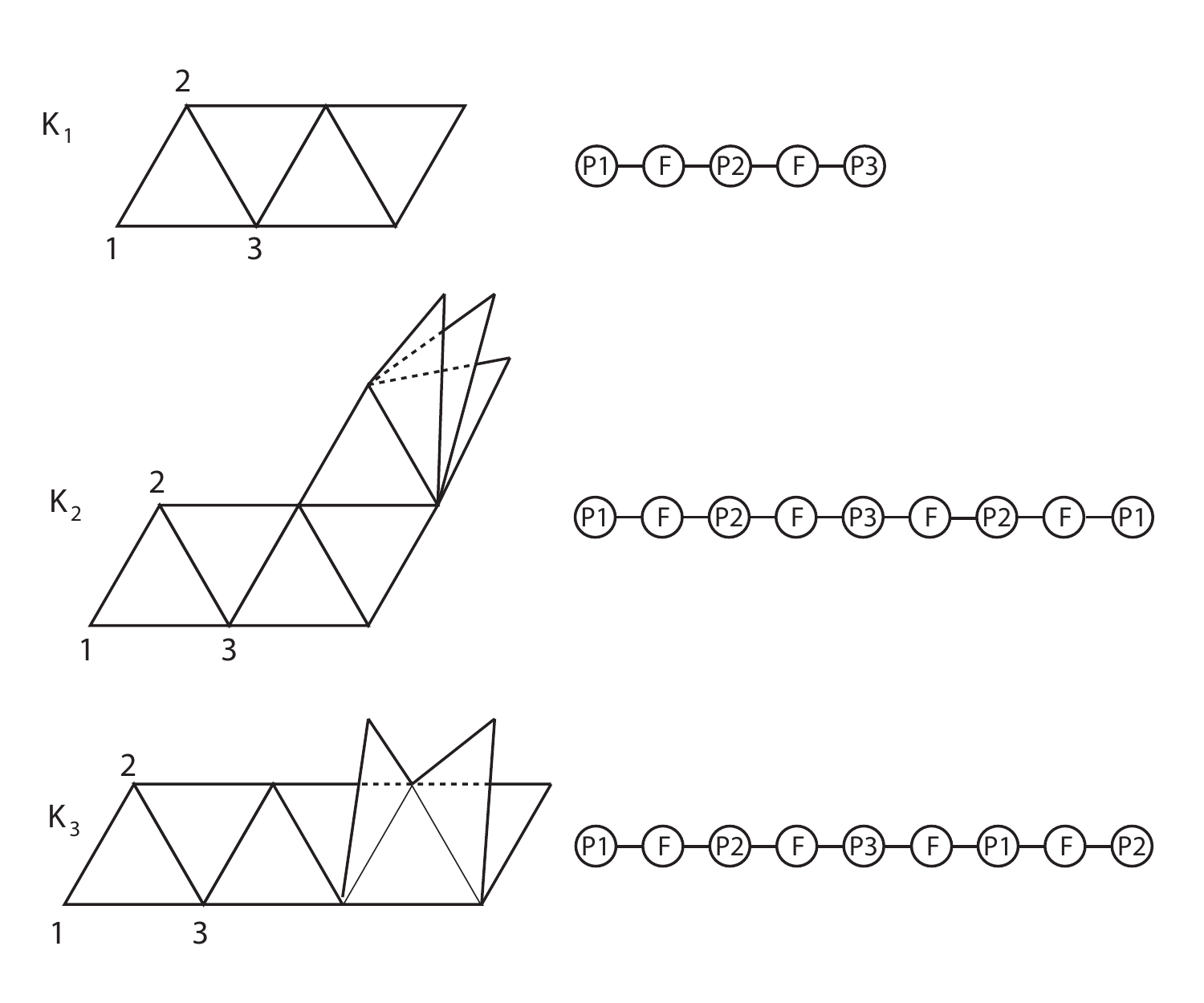}
    \caption{Three minimally branched complexes in $\mathcal T_{2}$ and their associated 
    labelled graphs.}
    \label{Figure 1}
  \end{figure}

  \begin{example}\label{example:max} 
      In Figure~\ref{Figure 2} we give two examples of 
      maximally branched complexes 
      $B_{1},B_{2}\in \mathcal T_{2}$ and their associated 
      labelled graphs $\Gamma(B_{i})$. 
      One can check by hand that 
      $\Gamma(B_{2})$ weakly covers $\Gamma(B_{1})$. In fact 
      $\Gamma(B_{1})$ is the minimal graph in this bisimilarity class. 
      Corollary \ref{maxbranchedntreesqi} generalizes this example.
  \end{example}

  \begin{figure}[ht]
      \centering
  \includegraphics[width=.7\hsize]{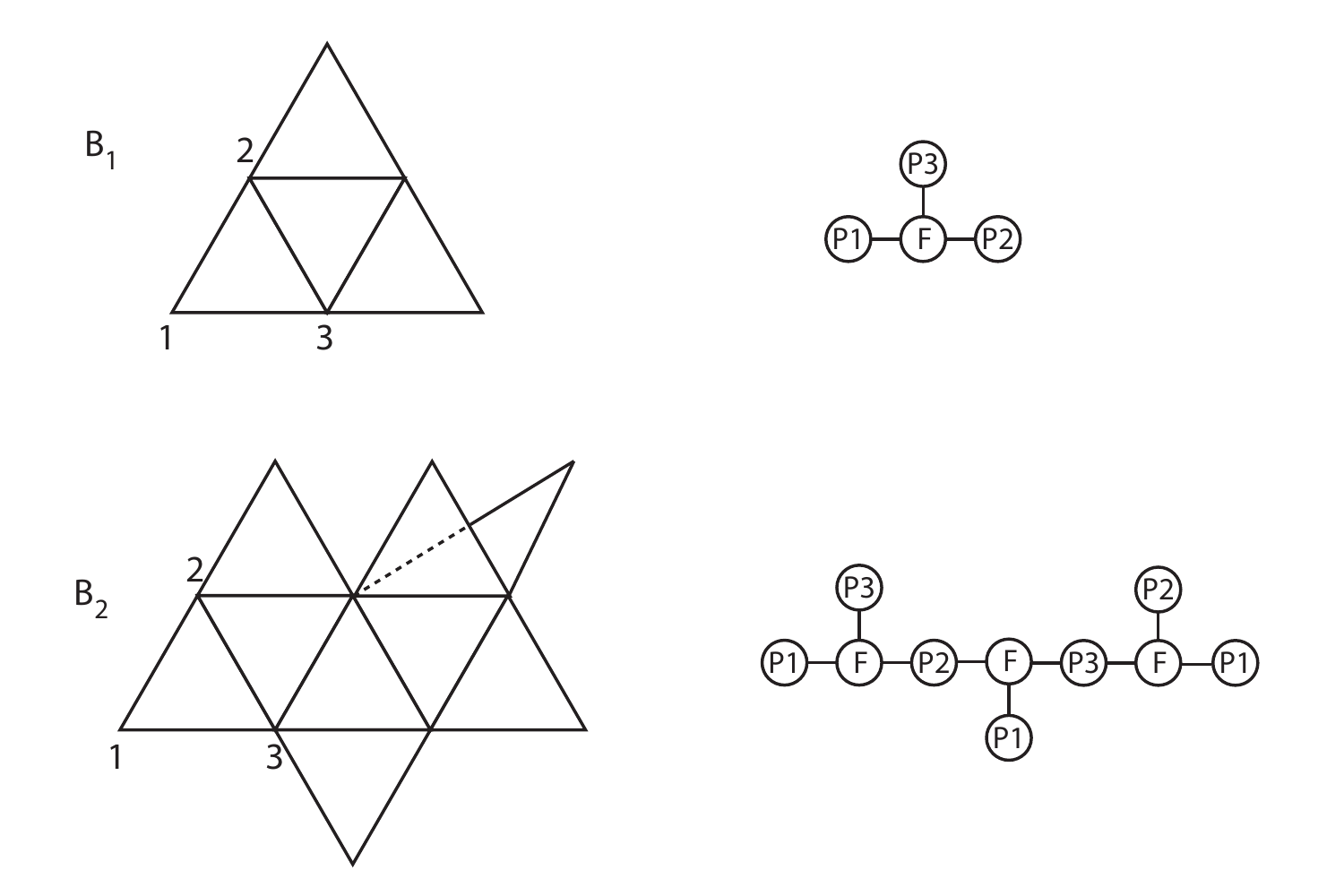}
      \caption{Two maximally branched complexes in $\mathcal T_{2}$ 
      and their associated labelled graphs.}
      \label{Figure 2}
    \end{figure}

    \begin{proof}[Proof of Corollary \ref{maxbranchedntreesqi}] The 
	minimal graph associated to an irreducible maximally branched 
	right-angled $n$--tree group is a star consisting of a single central
	  \f--vertex connected to $n+1$ \p--vertices, one of each color. 
    \end{proof}

    \begin{example} For each $n\geq 2$, any pair of complexes 
	$K, K'\in \mathcal T_{n}$ which use only two $\p$--colors 
	yield quasi-isometric groups and the minimal such graph,  
	up to permutation of labels, corresponds to a 
	graph of the form $\p1$---$\f$---$\p2$. 
	Any such group is reducible; 
	more generally, a group corresponding to a complex 
	$K\in\mathcal T_{n}$ is reducible if and only if its 
	graph uses less than $n+1$ $\p$--colors.
	See 
	Figure~\ref{Figure 3} for an example of a group in the 
	quasi-isometry class of such a $K\in \mathcal T_{2}$, but 
	whose graph is not minimal. 
	
	The number of minimal graphs with $k$ $\p$--vertices chosen 
	from a set of $3$ $\p$--colors grows with $k$. For $k=3$ 
	there are two minimal graphs, for $k=4$ there are three, 
	for $k=5$ there are twelve such graphs, and for $k=6$ there 
	are forty-five. 
	
	Note there are two
	quasi-isometry types corresponding to graphs with one $\p$--vertex, 
	one when the piece is just a simplex and the other when the piece 
	is built from more than one simplex, and just one minimal 
	graph with two $\p$--vertices, this is the 
	example given in Figure~\ref{Figure 3} 
	whose minimal graph is $\p1$---$\f$---$\p2$.
	Hence, there are exactly 65 quasi-isometry types of right-angled 
	$2$--tree groups built from $6$ or less pieces.
    \end{example}
    
    \begin{figure}[ht]
	\centering
    \includegraphics[width=.7\hsize]{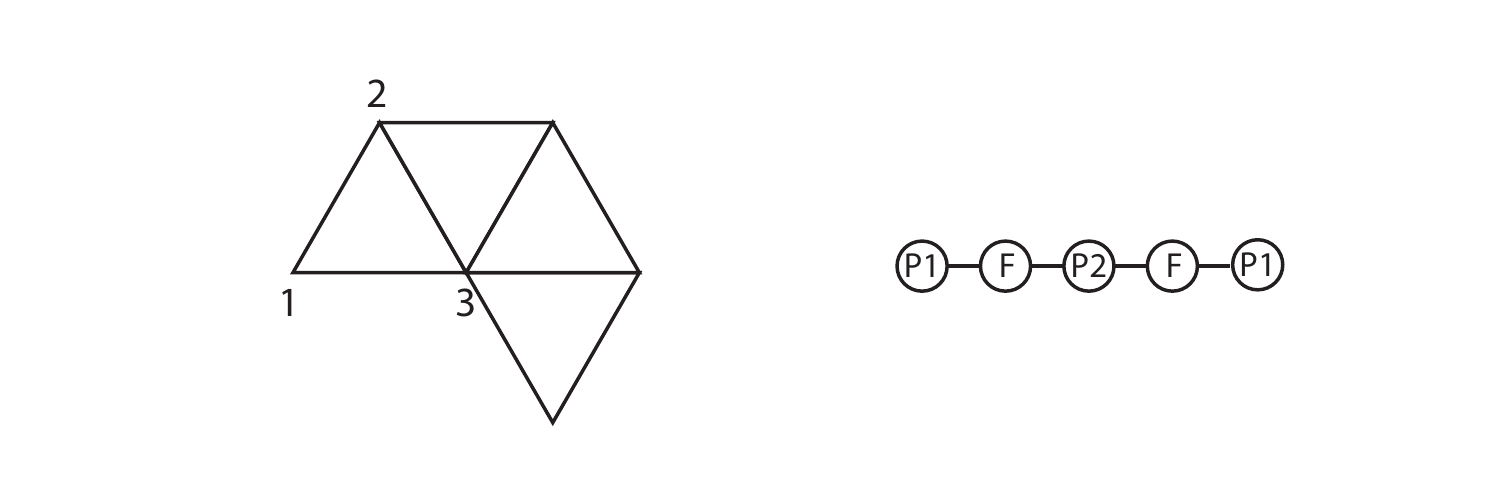}
	\caption{A complex in 
	$\mathcal T_{2}$ built from three 
	pieces and its associated labelled graph.}
	\label{Figure 3}
      \end{figure}

\section{Bisimilar implies quasi-isometric}\label{sec:3}

The following will prove the ``if'' direction of
Theorem~\ref{qintrees}.

\begin{theorem}\label{ntrees bisim implies qi}
    Fix $K,K'\in\mathcal T_n$. If the graphs 
    $\Gamma(K)$ and $\Gamma(K')$ are bisimilar, then $A_{K}$ and 
    $A_{K'}$ are quasi-isometric.
\end{theorem}

\begin{proof}[Proof of Theorem~\ref{ntrees bisim implies qi}]
  Fix a pair of complexes $K,K'\in \mathcal T_n$ for which $\Gamma(K)$
  and $\Gamma(K')$ are bisimilar and let $\Gamma$ denote the minimal
  graph in this bisimilarity class.
    
  Each group $A_K$ and $A_{K'}$ is represented as the
  fundamental group of the generalized ``graph space'' $X_K$ and
  $X_{K'}$ (it need not be a manifold, since it has up to $(n+1)$
  pieces glued together along each gluing torus), and is thus
  quasi-isometric to the universal cover of this space. Below we
  follow the same induction as in the proof of
  \cite[Theorem~3.2]{BehrstockNeumann:qigraph} to show that the
  universal covers of $X_{K}$ and $X_{K'}$ are bilipschitz
  homeomorphic, implying the quasi-isometry of $A_K$ and $A_{K'}$.

  The universal cover of a piece of $X_K$ or $X_{K'}$ is identified
with
  $\tilde{S_i}\times \R^{n}$, where $S_i$ is one of a finite
  collection of compact riemannian surfaces with boundary (each of
  these is a sphere minus a finite number, at least three, of open
  disks). Note that these $S_{i}$ all have bilipschitz homeomorphic
  universal covers.

Let $X_{0}$ denote the universal cover of a fixed riemannian metric on
a sphere minus three disks. Let $C$ be a finite set of ``colors''. A
bounded $C$-coloring on the boundary components of $X_0$ is an
assignment of a color in $C$ to each boundary component of $X_0$ such
that every point of $X_0$ is within a uniformly
bounded distance of boundary components of every color.  Choose a
fixed
boundary component of the universal cover, denoted
$\partial_{0}X_{0}$.
The following is Theorem
1.3 of \cite{BehrstockNeumann:qigraph}.
\begin{theorem}\label{th:BN}
  For any manifold $X$ bilipschitz
homeomorphic to $X_{0}$ with a bounded $C$--coloring on the
elements of $\partial X$, there exists $L$ and a function $\phi\colon
\R\to \R$ such that for any $L'$ and any color-preserving
$L'$--bilipschitz homeomorphism $\Phi_0$ from a boundary component
$\partial_0X$ of $X$ to $\partial_0X_0$, then $\Phi_0$ extends to a
$\phi(L')$--bilipschitz homeomorphism $\Phi\colon X\to X_0$ which is
$L$--bilipschitz on every other boundary component and which is a
color-preserving map from $\partial X$ to $\partial X_{0}$.\qed
\end{theorem}

Each piece of $K$ or $K'$ is associated with some \p--vertex $v$ of
the minimal graph $\Gamma$; we then say the piece has \emph{type $v$},
and similarly for the pieces of the geometric realizations $X_K$ and
$X_{K'}$ and their universal covers $\tilde X_K$ and $\tilde X_{K'}$.
We let $C_v$ denote the set of outgoing edges at the \p--vertex $v$,
so there is a natural $C_v$--labelling of the boundary components of
any type $v$ geometric piece of $\tilde X_K$ or $\tilde X_{K'}$.

Choose a number $L$ sufficiently large so that Theorem
\ref{th:BN} applies for the universal cover of each of the $S_{i}$.
Choose a bilipschitz homeomorphism from one type $v$ piece $\tilde
S_i\times \R$ of $\tilde X_{K}$ to a type $v$ piece $X_0\times \R^{n}$
of $\tilde X_{K'}$, preserving the (surface)$\times \R^{n}$ product
structure and the $C_v$--colors of boundary components; this can
be done since the graphs are bisimilar.  We want to extend to a
neighboring piece of $\tilde X_{K}$. On the common boundary $\R\times
\R^{n}$ we have a map that is of the form $\phi_1\times \phi_2$ with
$\phi_1$ and $\phi_2$ both bilipschitz. Since $\Gamma(K)$ and 
$\Gamma(K')$ are bisimilar, each neighboring piece in $\tilde X_{K'}$ 
has the same 
label as the corresponding piece in $\tilde X_{K}$ 
and thus we can extend over each neighboring piece 
by a product map. Further, by Theorem \ref{th:BN}, we can assume this map 
preserves boundary colors and on the other boundaries of this piece 
is given by maps of the form $\phi'_1\times \phi_2$ with $\phi'_1$
$L$--bilipschitz. We do this for all neighboring pieces of our
starting piece. Because of the flip, when we extend over the next
layer we have maps on the outer boundaries that are $L$--bilipschitz
in both base and fiber. We can thus continue extending outwards
inductively to construct our desired bilipschitz map.
\end{proof}

\section{Quasi-isometries preserve the decomposition into
pieces}\label{sec:4}

As described above, the group $A_{K}$ with $K\in\TTn$ is the
fundamental group of a ``graph space'' $X_K$ whose universal cover
$\tilde X_K$ is a quasi-isometric model for $A_K$. This $\tilde X_K$
has its geometric decomposition into \emph{pieces} which overlap each
other in \emph{separating flats}. (Equivalently, the same
decomposition is given directly on $A_K$ up to quasi-isometry by
writing $A_K$ as the union of the cosets of the \p--vertex groups of
its geometric graph-of-groups decomposition.) The asymptotic cone
$\coneA$ of $A_K$ (which equals the asymptotic cone of $\tilde X_K$) 
admits a similar decomposition into subsets coming from asymptotic
cones of the pieces of $A_K$, which we 
call \emph{pieces} as well; note that the asymptotic cone of any piece
is isometric to $T\times \R^{n}$ where $T$ is a metric tree (all the
asymptotic cones we consider are taken with respect to an arbitrary,
but fixed, choice of ultrafilter and scaling constants).  Below,
we apply Kapovich--Leeb's argument that quasi-isometries preserve the
geometric decomposition of $3$-manifolds \cite{KapovichLeeb:haken}, to
the present situation.

The following lemma in the case $n=1$ was proven in 
\cite[Lemma 2.14]{KapovichLeeb:nonpc}; the same argument holds to 
prove:
\begin{lemma}
    Fix a metric tree $T$. If $f\co\R^{n+1}\to T\times \R^{n}$ is a 
    bilipschitz embedding, then the image, $f(\R^{n+1})$, is a subset 
    which is isometric to $\R^{n+1}$.\qed
\end{lemma}

An immediate corollary of this lemma is that any subset of 
$\coneA$ which is contained in the asymptotic cone of 
one of the pieces and which is bilipschitz to $\R^{n+1}$ 
must actually be an isometrically embedded flat.

In a similar direction, the following also holds as in 
\cite[Lemma 3.3]{KapovichLeeb:haken}:
\begin{lemma}
    Let $T$ be a geodesically complete tree and $C\subseteq \R^{n}$ a 
    closed subset. If $f\co C\to T\times \R^{n}$ is a 
    bilipschitz embedding whose image separates, then $C=\R^{n}$ 
    and the projection of the image to $T$ is contained in a segment
with 
    no branch point in its interior. In particular, if $T$ branches 
    everywhere, then the image is a fiber $\{t\}\times\R^{n}$.\qed
\end{lemma}

The arguments of \cite{KapovichLeeb:haken} then apply 
to show that any bilipschitz embedding of a tree cross $\R^{n}$ into 
$\coneA$ must lie inside a piece, which then implies:

\begin{proposition} Let $K,K'\in \TTn$ and let $\coneA,\coneA'$ denote
  asymptotic cones of $A_K, A_{K'}$. Let $\phi\co\coneA\to\coneA'$
  be a bilipschitz homeomorphism. Then $\phi$ sends pieces to pieces
  and separating flats to separating flats.\qed
\end{proposition}

Using the \catzero structure on $\coneA$ and identical arguments as
for 
\cite[Theorem 4.6]{KapovichLeeb:haken}, one shows that any quasiflat 
which is not sublinearly close to a separating flat must diverge from 
it linearly and in particular that any quasi-isometry from $A_K$ to 
$A_{K'}$ sends flats to flats. As in \cite[Theorem 
1.1]{KapovichLeeb:haken}, this result applied in the present context  
implies the following theorem:
\begin{theorem}\label{geomstructurepreserved}
  Let $\phi\co A_K \to A_{K'}$ be a quasi-isometry. Then $\phi$
  preserves the geometric decompositions of $\tilde X_K$ and $\tilde
  X_{K'}$ in the following sense: for any geometric piece $X$ of
  $\tilde X_K$ there exists a geometric piece $X'$ of 
  $\tilde X_{K'}$ within a uniformly bounded Hausdorff 
  distance from $\phi(X)$. 
  Moreover, $\phi$ induces an isomorphism of trees dual to the 
  geometric decomposition of $\tilde X_K$ and $\tilde X_{K'}$.\qed
\end{theorem}

To complete the ``only if'' direction of Theorem~\ref{qintrees}, we
must show the induced map of trees also preserves the \p--vertex
labelling up to a permutation of labels, since this dual tree is then
the unique labelled tree in the bisimilarity class corresponding to
the associated Artin group.  To do this, it suffices to show that if
we know the geometric decomposition of $\tilde X_K$ up to
quasi-isometry then we can tell when two \p--vertices $v$ and $v'$ of
the decomposition tree have the same color. 

The path from $v$ to $v'$
consists of alternating \p-- and \f--vertices, 
$$v=v_0, w_1, v_1,\dots, w_r,v_r=v'\,,$$ 
say. The $T^{n}$ in the geometric piece for vertex $v_0$ defines an 
$n$--dimensional
sub-flat $\R^n$ of the flat $\R^{n+1}$ in $\coneA$ corresponding to
$w_1$. As we move along the path we intersect this subflat repeatedly
with the projection to this flat of the codimension--$1$ subflats 
for the vertices $v_1, v_2, \dots$ (alternatively, this can be 
interpreted as the coarse intersection of the subflats associated to 
the centers of the respective pieces; hence this is 
quasi-isometrically invariant since the center in a piece corresponds 
to the coarse intersection of all the maximal flats).
Whenever we pass a $v_i$ of a color we have not yet seen, the
dimension of the intersection drops by $1$. Otherwise, we know we have
already seen the color along the path, and by using the same procedure
to check backwards along the path from $v_i$, we can find which vertex
had the same color.  If it was not $v_0$ we then continue the same way
along the path. In this way we either show that $v_r$ has the same
color as $v_0$, or the dimension of our subspace has reached $0$ by
the time we get to $v_r$, in which case we have seen every color along
the path.  By induction we can assume that we have already determined
which $v_i$'s that are closer together along the path have the same
color. But then, by checking forwards along the path from $v_0$ and
backwards from $v_r$ we can tell that they both have different colors
from every other $v_i$ along the path, so must have the same color.

This completes the proof of the ``only if'' direction of
Theorem~\ref{qintrees}, and since the ``if'' direction is 
Theorem~\ref{ntrees bisim implies qi}, 
Theorem~\ref{qintrees} is proved.\qed

\begin{remark}\label{remark:labels} The above argument shows that 
	the construction of labels in 
	Section~\ref{sec: Labelled graphs} could be done purely  
	geometrically.  As an example to see this, compare the 
	example in Figure~\ref{B1 variant} to example B1 in 
	Figure~\ref{Figure 2}. To see how to label the shaded piece, 
	$N$, consider the associated space $\tilde X_{B_{1}'}$, and 
	note that in this space $N$ has 
	three-dimensional intersection with the $P3$ piece and 
	two-dimensional intersection 
	with both the pieces labelled $P1$ and $P2$. 
	The intersection of the torus associated to the 
	center of $N$ with the center of $P1$ is one dimensional, 
	while its intersection with the center of $P2$ is trivial. 
	Since a path starting from 
	the $P1$ or $P3$ piece would not need to traverse all the 
	labels, whereas a path starting at the $P2$ piece would have 
	traversed the $P2$ and $P3$ labels, we thus see that 
	the $N$ must be labelled by a 1, as given by the labelling in 
	Section~\ref{sec: Labelled graphs}.
    
    \begin{figure}[ht]
	\centering
    \includegraphics[width=.7\hsize]{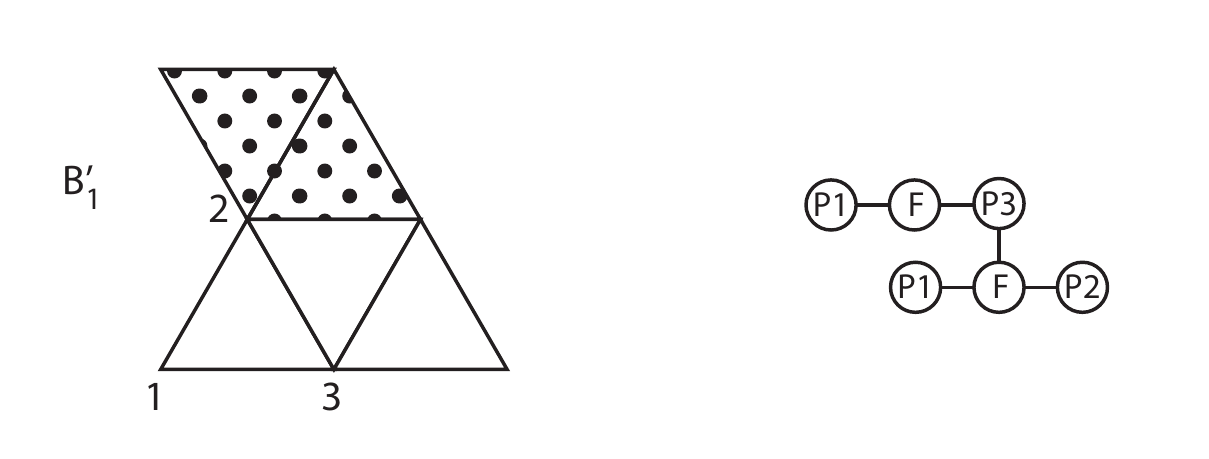}
	\caption{Compare with example B1 in 
	Figure~\ref{Figure 2}.}
	\label{B1 variant}
      \end{figure}

\end{remark}

\def\cprime{$'$}

\end{document}